\documentclass[10pt|11pt|12pt]{article}
\usepackage{cite}
\usepackage{amssymb,amsmath,latexsym,mathrsfs}
\usepackage{psfrag}
\usepackage{bm}
\usepackage{cite}
\usepackage{url}
\usepackage{color}
\usepackage{graphicx}
\usepackage{caption}
\usepackage{subcaption}
\usepackage{booktabs}
\usepackage{multirow} 
\usepackage{mathtools} 
\usepackage{extarrows} 
\usepackage{centernot}

\usepackage{authblk}

\usepackage{cite}
\usepackage{amsthm}
\newtheorem{defi}{Definition}
\newtheorem{theorem}{Theorem}[section]
\newtheorem{lemma}[theorem]{Lemma}
\newtheorem{corollary}[theorem]{Corollary}

\newtheorem{remark}{Remark}

\newtheorem{example}{Example}

\newcommand{\La}{  { \Lambda }}
\newcommand{\1}{  {\tilde \omega_{\La^c} }}
\newcommand{\mm}{  { \mu_{\phi, \La, \1}  }}

\usepackage{datetime}

\newdate{date}{11}{04}{2018}
\date{\displaydate{date}}

\begin{document}

\title{\LARGE \bf Random-cluster correlation inequalities for 
 Gibbs fields }

\author{ Alberto Gandolfi }

\affil{NYU Abu Dhabi}
\affil{Dipartimento di Matematica e Informatica U. Dini\\
 Universit\`a di Firenze}

\vspace{5mm}

\maketitle

\thispagestyle{plain}
\pagestyle{plain}

\begin{abstract}

\end{abstract}

In this note we prove a correlation inequality for local variables of a Gibbs field
 based on the connectivity
by active hyperbonds
in a random cluster representation of the non overlap configuration distribution
of two independent copies of the field.

As a consequence, we show that absence of Machta-Newman-Stein
blue bonds percolation implies uniqueness of Gibbs distribution 
in EA Spin Glasses. In dimension two this could constitute a step towards a
proof  that the critical temperature is zero.

%In this note we generalize the FK bound for
%the Ising model correlation by providing a  correlation  inequality for a binary Gibbs distribution
%based on a particular random cluster connectivity.
%Specifically, we show that correlations of any two local observables
%are bounded  by an active hyperbond 
%connectivity,  averaged over the suitably defined overlap
%distribution, 
%in a  random cluster representation of the non overlap
%configuration distribution between two independent copies
%of the Gibbs field.
%
%As a byproduct, we get a sufficient condition for the uniqueness of the Gibbs distribution
%which allows to reduce the long standing conjecture
%about the zero  temperature critical point
%in the two-dimensional Edwards Anderson Spin Glasses
%to the absence  of 
% Machta-Newman-Stein blue bonds percolation.

%\tableofcontents


\section{Introduction} 

The classic FK representation \cite{FK72} allows to express the spin spin correlation of 
Ising ferromagnets in terms of the FK occupied bonds connectivity
in the random cluster representation of the Ising model. No similar expressions or 
bounds 
are available for correlation of more general observables, or even
for spin correlations in
general Gibbs models.

In this paper, we provide a  general correlation  inequality of this type
by taking two copies of the Gibbs field. Specifically,
the correlation  of any two local observables
is bounded by  the active hyperbond 
(a generalization of the occupied bonds) connectivity 
in a (generalized) random cluster representation of the non overlap
configuration distribution between two independent copies
of a Gibbs field (called "foldings"  in \cite{BG13}),
the connectivity being  averaged over the overlap configuration distribution.
It is, to the knowledge of the author, the first result of this type with such a wide
spectrum of applicability.

The overlap and non-overlap configuration distributions in two independent copies 
are obtained, for 
a  general Gibbs measure, by fixing the value of the sum of the corresponding spins for each vertex: 
we declare an overlap if there is only one pair of spin values,
one spin from each copy, whose sum is
equal to the prescribed sum, and a non-overlap when there are more possibilities.
For the special case in which  the spins take only the values $\pm 1$, as in Ising model or
Spin Glasses, we have overlap when the two spins agree, and non-overlap when
they disagree, in accordance with standard terminology about overlap in Spin Glass theory  
\cite{MNS08}.

The use of  the sum (and difference) of spins in two independent copies
of an Ising model
 traces back to Percus and Lebowitz  \cite{L74}, and the non-overlap configuration distribution
 is  essential in \cite{R00}. A 
generalized random cluster representation (RCR) appears in
\cite{BG13}, and it is further generalized here. 
The use of product measure to prove inequalities appears also in 
\cite{G18}.

Some consequences of the main inequality are 
then drawn in the paper; they include a criterium for extremality of a Gibbs 
distribution, which allows to retrieve  
 the critical point for uniqueness of Markov Chains
 in a ferromagnetic Ising model on a binary tree; and a criterium for uniqueness
of the Gibbs distribution, which is then compared with Dobrushin criterium and 
disagreement percolation (see Section $4$).

Our results are  then related, in Section $4.5$, to the MNS representation of $\pm J$ spin glasses
 \cite{MNS08}; our final result is that absence of percolation of MNS blue bonds
 (in the non overlap region, and, hence, tout court absence of such percolation)
 for any boundary condition implies a.e. uniqueness of the Gibbs distribution.
 It is then conceivable that two dimensional geometric constraints prevent
 the formation of MNS blue bond percolation at any finite temperature, yielding a proof of the 
 long standing conjecture that phase transition occurs at zero temperature
 in the two dimensional EA model (see \cite{NS13}, Page 84, or \cite{TTC17}, Page 48).

\bigskip

The author would like to thank J. van den Berg, C. Newman and D. Stein for very valuable discussions. 

\section{Preliminaries}

\subsection{Overlap and non overlap configuration distributions} 

We consider an infinite graph  $\mathcal G = (V, E)$, and a locally finite 
family of hyperbonds 
$\mathcal B=\{b \subset V, b \text{ finite}\}$; $\La$  indicates
 finite subsets of $V$, and $\delta \La$ is the set of vertices $v \in V$ such 
that there is a $b \in  \mathcal B$ with $v \in b, b \cap \La \neq \emptyset,
b \cap \La^c \neq \emptyset$.

We then take
$\Omega= F^V$, with $F$ a finite alphabet; and for any  $\Lambda
\subset V$, $\Omega_{\Lambda} = F^ \Lambda$.
Combined configurations
are denoted by juxtaposing symbols, so for disjont $C_1, C_2$,
 $\omega_{C_1} \omega_{C_2} $ is the configuration of $\Omega_{C_1 \cup C_2}$
 obtained from $\omega_{C_1}$ and $ \omega_{C_2} $.

We then consider an interaction $\phi$ defined on $\cup_{b \in \mathcal B} \Omega_b$; 
to include possible 
constraints we allow $\phi = \infty$; although it would be more expressive to use a different 
collection of hyperbonds from $\mathcal B$ for possible constraints (see \cite{GL16})),
such a distinction is not needed for our purposes here.

The $\mathcal B${\bf-Gibbs measure} on $\La$ with interactions $\phi$ 
and boundary conditions $\1$ is
$$
\mu_{\phi, \La, \1} (\omega_{\La})=
\frac{1}{Z} e^{\sum_{b \in \La} \phi(\omega_b) + 
\sum_{b: b \cap \La \neq \emptyset, b \cap \La^c \neq \emptyset}
\phi(\omega_{b \cap \La} \tilde \omega_{b \cap \La^c})}.
$$
Thermodynamic limits  in $\La$ of $
\mu_{\phi, \La, \1}$ are denoted by
$\mu_{\phi}= \mu_{\phi, \tilde \omega} $, where $\tilde \omega$
denotes a sequence of boundary conditions.

\bigskip
We next consider  two copies $\La^{(1)}, \La^{(2)}$ of $\La$,
and the product  space $\Omega'=
\Omega_{\La^{(1)} } \times \Omega_{\La^{(2)} }$ with the
product measure $\mu_{\phi, \La, \1}^{(1)} \times \mu_{\phi, \La, \1}^{(2)}$,
with $\mu_{\phi, \La, \1}^{(\ell)} =^d \mu_{\phi, \La, \1}$, i.e. we consider
two independent copies of the model.

We are  interested in conditioning to the values of local functions
with some invertibility. There are many possible choices,
and for simplicity we restrict to the sum of spins, i.e. to
$\sigma_i=\omega^{(1)}_i + \omega^{(2)}_i$,
where $\omega^{(\ell)} \in \Omega^{(\ell)}$. 
Then, $\sigma=\{\sigma_i \}_{i \in \La} \in \Sigma= \prod_{i \in \La} \tilde F$
with $\tilde F=\{a: a= a_1+a_2, a_{\ell} \in F\}$.
Given $\sigma \in \Sigma$ and $\omega \in \Omega$,
we denote by $\sigma - \omega$ the configuration
such that $(\sigma - \omega)_i=\sigma_i - \omega_i$. Given
$\sigma$, a probability
$\mu$ on $\Omega$ is called {\bf $\sigma$-symmetric}
if $\mu(\omega)=\mu(\sigma - \omega)$ for all $\omega$'s.

Given $\sigma$, let $W_{ \sigma}=\{ (\omega^{(1)}, \omega^{(2)}):
\omega^{(1)}_i + \omega^{(2)}_i = \sigma_i \text{ for all } i \in \La\}$.
The collection 
$
\{W_{ \sigma}\}_{ \sigma \in \Sigma}
$
forms a partition of $\Omega'$, and we define
the  {\bf  overlap configuration distribution}
$$
 \rho_{\phi, \La, \1}(\sigma)
 =(\mu_{\phi, \La, \1}^{(1)} \times \mu_{\phi, \La, \1}^{(2)}) 
 (W_{ \sigma}),
 $$
and 
the {\bf non overlap configuration distribution} 
$$
 \mu^{ \sigma}_{\phi, \La, \1}(\omega)
 =(\mu_{\phi, \La, \1}^{(1)} \times \mu_{\phi, \La, \1}^{(2)}) 
 ((\omega,\sigma-\omega) | \sigma).
 $$
 By definition, for an event $A \subseteq \Omega$, 
\begin{eqnarray}
\mu_{\phi, \La, \1}(A)&=&
(\mu_{\phi, \La, \1}^{(1)} \times \mu_{\phi, \La, \1}^{(2)})(A \times \Omega)
\nonumber \\
&=&\sum_{\sigma \in \Sigma} \mu^{\sigma}_{\phi, \La, \1}(A) 
\rho_{\phi. \La, \1} (\sigma).
\end{eqnarray}

In the special case that $\Omega=\{-1,1\}^{\Lambda}$, $\sigma_i \in \{-2,0,2\}$;
moreover, conditioned on $\sigma_i$, $\sigma_i-\omega_i $ equals $ -\omega_i$
if $\sigma_i=0$, and $+\omega_i$ otherwise. The overlap region is where
$\sigma_i \neq 0$, as then only the pair $(\omega_i^{(1)}, \omega_i^{(2)})$
is admissible with $\omega_i^{(j)}= \sigma_i/2$ for both $j$'s;
and the nonoverlap region is where $\sigma_i = 0$, in which case two pairs are admissible.
This example
 justifies the reference to overlap in the names given to the two  distributions above.

 Given $\sigma$, let $\Omega(\sigma) = \prod_{i \in \Lambda}
 F(\sigma_i) \subseteq \Omega$, where $F(\sigma_i)=\{a \in F: \exists b \in F \text{ with }
 a+b=\sigma_i\}$. The region $K(\sigma)=\{i:|F(\sigma_i)|=1\}$ is the the  overlap region,
 and the non overlap distribution is, in fact, a distribution on 
 $K^c(\sigma)$:
in \cite{BG13} such distribution is called a "folding" of $\mu$, and foldings
form a collection of distributions indexed by the overlap configuration $\sigma$.
If the initial distribution is Gibbs, then the non overlap distribution is  
simmetrized  Gibbs:
  \begin{lemma} \label{2.100}
 Given a $\mathcal B( \La)$-Gibbs distribution 
 $\mu_{\phi, \La, \1}$ on $\Omega$ with interactions $\phi$,
 and given $\sigma \in \Sigma$, the non overlap configuration distribution
 $\mu^{\sigma}_{\phi, \La, \1}$ is a $\sigma$-symmetric
 $\mathcal B( \La)$-Gibbs distribution on $\Omega(\sigma)$
 with interactions
 $
 \phi'(\omega_b)= \phi(\omega_b)+\phi(\sigma_b-\omega_b)$
 and boundary conditions $\1$.

\end{lemma}

\begin{proof} Given $\sigma$, for $\omega \in \Omega(\sigma)$
we have
\begin{eqnarray} \label{3}
  \mu^{ \sigma}_{\phi, \La, \1}(\omega)
&=&(\mu_{\phi, \La, \1}^{(1)} \times \mu_{\phi, \La, \1}^{(2)}) 
 ((\omega,\sigma-\omega) | \sigma)\nonumber \\
 &=&
 \frac{1}{Z'} \exp(\sum_{b \in \La} ( \phi(\omega_b) + \phi(\sigma_b-\omega_b) ) \nonumber \\
 && \hskip 0.5 cm
+\sum_{b: b \cap \La \neq \emptyset, b \cap \La^c \neq \emptyset}
 \phi(\omega_{b \cap \La} \tilde \omega_{b \cap \La^c})
+\phi(\sigma_{b } -\omega_{b \cap \La}  \tilde \omega_{b \cap \La^c} )\nonumber \\
 &=&
 \frac{1}{Z'} \exp(\sum_{b \in \La}  \phi'(\omega_b)  \nonumber \\
 && \hskip 0.5 cm
+\sum_{b: b \cap \La \neq \emptyset, b \cap \La^c \neq \emptyset}
 \phi'(\omega_{b \cap \La} \tilde \omega_{b \cap \La^c}
 )   
\nonumber  \\
 &=&  \mu^{ \sigma}_{\phi, \La, \1}(\sigma-\omega ).
\nonumber
\end{eqnarray}
 Clearly, $ (\mu_{\phi, \La, \1}^{(1)} \times \mu_{\phi, \La, \1}^{(2)}) 
 ((\omega,\sigma-\omega))=0$ for $\omega \notin \Omega(\sigma)$.

\end{proof}

\subsection{Random cluster representations RCR's}

%
%
%For clearness, we list now the various probabilities
%we deal with in this work, together with their counterpart in 
%the 
% standard FK representation of the ferromagnetic Ising model. 
%\bigskip
%
%\begin{tabular}{|  c  | c | c | c | }
%
%  \hline                        
% Name & Symbol& Role & FK case \\
%  \hline 
%  \multirow{2}{6em}{Initial probability } & $ \mu$  & Probability  &  Ising model $\mu_{J,0}$\\ 
%&  &  to be represented&\\
%  \hline 
% \multirow{2}{6em}{Base probability} & $\nu$  & Starting measure & Bernoulli\\ 
%&  &  of the representation&with parameter $p$ \\
%  \hline 
%   \multirow{3}{6em}{Joint distribution } &  & Joint distribution&Joint \\ 
%& $Q$  &  of spin and  & spin and FK bond \\
%&  &   (hyper)bond variables & distribution \\
%   \hline 
%   \multirow{3}{8em}{Random Cluster Probability} & & Marginal on & FK measure $\phi_{p,2}$\\ 
% & $P$   &  hyperbond &with parameter $p$ \\
%&  &  variables&and  $q=2$ \\
%  \hline  
%   \multirow{3}{8em}{Overlap configuration distribution } &   & Distribution  & Developed
%   \\ 
%& $ \rho $ &  on spin sums& in this paper \\
%&& $ \sigma \in \Sigma$ &\\
%  \hline 
%    \multirow{3}{8em}{Non-overlap configuration distribution } &   & Conditional distribution  & 
%   \\ 
%& $ \mu^{\sigma} $ &  given $\sigma$& Folding\\
%&&   &\\
%  \hline 
%   \multirow{3}{8em}{Integrated Random Cluster Probability } &   &   & Developed
%   \\ 
%& $ \overline P$ & $\rho$-average  of $P$  &in this paper\\
%&&  &\\
%  \hline 
%\end{tabular}
%
%\bigskip
%Here are the detailed definitions. 

\medskip

A  {\bf $\cal B$-Random Cluster Representation} or {\bf $\cal B$-RCR} of a probability $\mu$ on 
$\Omega_\La$ is a way of expressing $\mu$ 
using a {\bf base probability} $\nu$ on 
the configurations  $\eta \in H =
 \prod_{b \in \mathcal B} \mathcal P(\Omega_b)$,
 where $\mathcal P(\Omega_b)$ indicates the collection of all subsets of $\Omega_b$,
  as follows:
 for every $\omega \in  \Omega_\La$
\begin{eqnarray} \label{2.31}
\mu(\omega) = \frac{1}{Z_1} \sum_{\eta \in H} \nu(\eta) \prod_{ b \in \mathcal B}
\mathbb I_{\omega_b \in \eta_b} =  \frac{1}{Z_1} \sum_{\eta \in H, \eta \sim \omega} \nu(\eta) ,
\end{eqnarray}
where $Z_1$ is a normalizing factor. 
Configurations in $H$, namely prescriptions of collections of local 
spin configurations, are called {\bf hyperbond variables}, or simply 
bond variables when only pairs of spins are involved.

Note that the $\cal B$-RCR is not 
unique. Note also that \eqref{2.31} can be used to define new probabilities $\mu$
once $H$ and $\nu $ are given. 

\bigskip

If $\mu=\mu_{\phi, \La, \1}$ is $\cal B$-Gibbs, with possible hard core constraints,
 then the following procedure produces a variety of $\cal B$-RCR's
in which $\nu= \prod_{b \in \cal B} \nu_b$ is Bernoulli. For each $b $, consider 
the "energy" levels $\{e^{\phi(\omega_b)}\}_{\omega_b \in \Omega_b}=
\{e^{\phi_1}, \dots, e^{\phi_k}\}$, $\phi_1>\dots >\phi_k$, and 
a collection of $k$ subsets $\eta^{(1)}_b, \dots, \eta^{(k)}_b$ of $\Omega_b$
which, for simplicity, do not split energy levels, i.e. $\omega_b \in  \eta^{(i)}_b $ and $
\phi(\omega_b)=\phi(\omega_b')$ imply $\omega_b' \in  \eta^{(i)}_b $
(notice that for mathematical convenience there is a plus sign in 
the exponent).
Then let $A = [a_{i,j}]$ be a $k\times k$ $0$-$1$  matrix with $a_{i,j}$
the indicator that $ \eta^{(j)}_b $ contains all the configurations with energy level
$\phi_i$. If, for all $b$'s, the problem 
\begin{eqnarray} \label{2.1} 
A \vec p = c_b [e^{\phi_i}] \text{ for some  } c_b \in \mathbb R, 
\quad \sum_{j=1}^k p_j=1, p_j \geq0, 
\end{eqnarray}
where $ \vec p$ is the vector with components $p_j's$,
can be solved, then we can 
take  $\nu_b (\eta^{(j)}_b):= p_j$
and $\nu$ is the base of a $\cal B$-RCR of $\mu_{\phi, \La, \1}$.

One particular case of the above mechanism has been described in \cite{BG13},
in which $ \eta^{(i)}_b $ are taken to be monotone, i.e. 
$\omega_b \in  \eta^{(i)}_b $ iff $\phi(\omega_b) \geq \phi_i$. 
In that case, $A$ is $0$-$1$ upper triangular; as the vector $ [e^{\phi_i}]$ is 
monotone decreasing,  there is a nonnegative solution to the above problem,
which can be normalized to give $\nu_b (\eta^{(i)}_b)
= \frac{e^{\phi_i}-e^{\phi_{i+1}} }{ e^{\phi_1}}$, with $\phi_{k+1}=-\infty$.

\subsection{Typed RCR and MNS blue-red bonds} \label{S2.3}

To achieve some additional expressive power, a typed RCR can also be introduced,
in which different types of hyperbond configurations are used. To keep things
simple, we discuss the case of two types only. A two-typed $\cal B$-RCR of a probability $\mu$
is given  by a pair of probabilities $\nu^{(a)}$ on $H^{(\alpha)}= H$ and $\nu^{(\beta)}$ on $H^{(b)}=H$
such that for every $\omega \in  \Omega_\La$
\begin{eqnarray} \label{2.3}
\mu(\omega) = \frac{1}{Z_2} \sum_{\eta \in H} \nu^{(\alpha)}(\eta)  \nu^{(\beta)}(\eta) \prod_{ b \in \mathcal B}
\mathbb I_{\omega_b \in \eta^{(\alpha)}_b} \mathbb I_{\omega_b \in \eta^{(\beta)}_b} ,
\end{eqnarray}
where $Z_2$ is a normalizing factor; the above definition introduces some novelties when
constraints are imposed on the possible values of $\eta^{(\alpha)}$ and $\eta^{(\beta)}$. The  constrained linear problem \eqref{2.1} becomes then
\begin{eqnarray} \label{2.4}
[A^{(\alpha)} \vec p^{(\alpha)}]_k [A^{(\beta)} \vec p^{(\beta)}]_k = c_b [e^{- \phi_k}]
\text{ for some constant }c_b,  \nonumber\\
\sum_{j=1}^k p^{(\alpha)}_j=\sum_{j=1}^k p^{(\beta)}_j=1; \quad  p^{(\alpha)}_j,p^{(\beta)}_j \geq0, 
\end{eqnarray}
with possibly additional  constraints on the entries of $p^{(1)}$ and $p^{(2)}$.

\bigskip 

An example of typed RCR appears in \cite{MNS08}, with a different terminology,  for the quenched distribution
of two independent copies of EA Spin Glasses. In each copy, $V=\mathbb Z^d$,
$\mathcal B(V)$ consists of n.n. pairs of vertices,
 $\{-1,1\}^V$  and 
\begin{eqnarray} \label{2.31}
\mu_{{\bf J}, \La, \1}(\omega_\La)= \frac{1}{Z} e^{\sum_ {\{ i,j \}} J_{i,j} \omega_i \omega_j
+ \sum_ {\{ i,j \}, i \in \La, j \not \in  \La} J_{i,j}\omega_i \tilde\omega_j} ,
\end{eqnarray}
where $J_{i,j}$ are i.i.d. symmetric r.v.'s taking values in $\pm J$, for 
some fixed $J>0$. The model is quenched, in the sense that a fixed value of 
${\bf J} = \{J_{i,j} \}_{\{i,j\} \in \mathcal B(V)}$ is taken, and later averaged on the
$\bf J$'s.
We consider then two copies of the space: $\La= \La^{(1)} \times \La^{(2)}$, two identical copies; $\omega_\La= \omega_{\La^{(1)}}
\times \omega_{\La^{(2)}}$, the product of any pair of not necessarily identical configurations; $b =b^{(1)} \times b^{(2)}= \{i,j\} \times \{i,j\}$, two copies of the same bond; and
$$
\phi(\omega_b) = J_{i,j} (\omega^{(1)}_i \cdot \omega^{(1)}_j + \omega^{(2)}_i \cdot \omega^{(2)}_j),
$$
where  we have indicated by a dot 
the actual products of the values of the two spin configurations.

For the product space above, one can produce a 
 one typed  RCR from energy levels  as follows.
The energy levels are $\phi_1=2 J_{i,j}, \phi_2=0, \phi_3=-2 J_{i,j}$; if 
$\Omega_{\phi_i} =\{\omega_b : \phi(\omega_b) = \phi_i\}$,
then $\eta_b$ can take one of the values $\Omega_{\phi_1},
\Omega_{\phi_1} \cup \Omega_{\phi_2}$ or $ \Omega_b$; and
$$
A= \begin{bmatrix}
1&1&1 \\
0&1&1\\
0&0&1
 \end{bmatrix} , \quad \quad \nu_b(\eta_b)= \begin{cases} 
1-e^{-2 J_{i,j}} \quad \text{ if  } \eta_b=\Omega_{\phi_1} \\
e^{-2 J_{i,j}}- e^{-4 J_{i,j}} \quad \text{ if  } \eta_b= \Omega_{\phi_1} \cup \Omega_{\phi_2}\\
e^{-4 J_{i,j}} \quad \text{ if  } \eta_b= \Omega_b .
 \end{cases} 
 $$

 On the other hand, a two typed RCR can be obtained constraining $\nu^{(\alpha)}$
 to single out only energy levels corresponding to
  two configurations $\omega_{\La^{(1)}}$
 and $\omega_{\La^{(2)}}$ which agree with the coupling in both copies,
 i.e. such that
 $J_{i,j} \omega^{(1)}_i \omega^{(1)}_j=J_{i,j}  \omega^{(2)}_i \omega^{(2)}_j=1$;
 and $\nu^{(\beta)}$ to single out only the energy level corresponding to agreement with the
 coupling in exactly one of the two copies, i.e. such that
  i.e. $\omega^{(1)}_i \omega^{(2)}_i  \omega^{(1)}_j \omega^{(2)}_j=-1$.
 In this case, $
A^{(1)}= \begin{bmatrix}
1&1 \\
0&1\\
0&1
 \end{bmatrix} $ and $
A^{(2)}= \begin{bmatrix}
0&1 \\
1&1\\
0&1
 \end{bmatrix} $;  and \eqref{2.4} becomes
 $$
 \begin{bmatrix}
(p_1^{(\alpha)}(i,j)+  p_2^{(\alpha)}(i,j) )p_2^{(\beta)}(i,j)  \\
p_2^{(\alpha)}(i,j)(p_1^{(\beta)}(i,j)+  p_2^{(\beta)} (i,j))\\
p_2^{(\alpha)} (i,j) p_2^{(\beta)}(i,j)
 \end{bmatrix} 
= c  \begin{bmatrix}
e^{2 J_{i,j}}  \\
0 \\
e^{-2 J_{i,j}}
 \end{bmatrix} 
 $$
The only solution
 is $p_1^{(\alpha)}(i,j)=1-e^{-4 J_{i,j}}, p_1^{(\beta)}(i,j)= 1-e^{-2 J_{i,j}}$, as
indicated in \cite{MNS08}; the two variables are called there blue and
red bonds, respectively, each being present with probability $p_1^{(\alpha)}$
and
$p_1^{(\beta)}$, respectively.
This is a two-typed RCR of the $\mathcal B(\Lambda)$-Gibbs distribution of 
two independent quenched EA Spin Glass configurations as
for $\mathbb I^{(\alpha)}=\mathbb I_{ J_{i,j} \omega^{(1)}_i \omega^{(1)}_j=1} \mathbb I_{J_{i,j}  \omega^{(2)}_i \omega^{(2)}_j=1}$
and $\mathbb I^{(\beta)}=\mathbb I_{  \omega^{(1)}_i \omega^{(1)}_j \omega^{(2)}_i \omega^{(2)}_j=-1}$
we have, with no boundary conditions,
\begin{eqnarray*}
&& \sum_{(\eta^{(\alpha)}, \eta^{(\beta)}) \sim (\omega^{(1)},\omega^{(2)})}
\nu^{(\alpha)}(\eta^{(\alpha)}) \nu^{(\beta)}(\eta^{(\beta)}) \\
&& \quad \quad \quad =
\frac{1}{Z} \prod_{\{i,j\} \in \mathcal B(\Lambda)} p_2^{(\alpha)}(i,j)
(1-  \mathbb I^{(\alpha)})
p_2^{(\beta)}(i,j)(1- \mathbb I^{(\beta)}  )\\
&& \quad \quad \quad =
\frac{1}{Z}e^{ -4 \sum_{\{i,j\}: \mathbb I^{(\alpha)}=1} J_{i,j}
-2 \sum_{\{i,j\}: \mathbb I^{(\beta)}=1} J_{i,j} }
\\
&& \quad \quad \quad 
= \frac{1}{Z} e^{\sum_ {\{i,j\} \in \mathcal B(\Lambda)} J_{i,j} \omega^{(1)}_i \omega^{(1)}_j+
J_{i,j} \omega^{(2)}_i \omega^{(2)}_j }\\
&& \quad \quad \quad 
=(\mu_{{\bf J} , \La}\times \mu_{{\bf J}, \La})(\omega^{(1)}_\La,\omega^{(1)}_\La)
\end{eqnarray*}
where the equality before the last one is obtained by factoring out $\prod_{\{i,j\} \in \mathcal B(\Lambda)}
2 J_{i,j}$. Boundary conditions can easily be incorporated.

The bond variables $\eta^{(\alpha)}$ are called blue bonds, and the $\eta^{(\beta)}$
red bonds, in the MNS representation.

\bigskip

\subsection{Active hyperbond connectivity}

In a RCR, the {\bf joint distribution on spin and hyperbond variables} is 
denoted by 
\begin{eqnarray} \label{2.32}
Q_{\phi, \La, \1}(\omega, \eta)= \frac{1}{Z_1}  \nu(\eta) \prod_{ b \in \mathcal B}
\mathbb I_{\omega_b \in \eta_b} =  \frac{1}{Z_1}  \nu(\eta) \mathbb I_{\omega \sim \eta}.
\end{eqnarray}
Then, 
the {\bf Random Cluster Probability $P=P_{\phi, \La, \1}$} is 
the marginal on the hyperbond variables:
$P_{\phi, \La, \1}(\eta)= \sum_{\omega \in \Omega} Q_{\phi, \La, \1}(\omega, \eta).$
Notice that random cluster probability is absolutely continuous w.r.t. $\nu$,
with a Radon Nikodym derivative  computable, in principle, in terms of the geometrical features 
which can described in terms of $\eta$ (this is where the factor $2^{\text{ number of clusters }}$ 
 appears in the original FK distribution).

The most relevant feature of the hyperbond configuration $\eta_b$ at $b$ is
whether it puts some restrictions on the compatible configurations $\omega_b$ or
not.
Given $\eta \in H$, the hyperbonds $b$ for which $\eta_b \neq \Omega_b$
are called  {\bf active}. We say that two hyperbonds $b(1), b(2)$ are directly connected
if $b(1) \cap b(2) \neq \emptyset$; and that two sets of vertices $\La_1, \La_2$ are connected
by active hyperbonds if there is a chain of sequentially directly connected
active hyperbonds, two of which have non empty intersections with  $M_1, M_2$.
We indicate this event by $\La_A 
\xleftrightarrow {act} \La_B$. Connected active hyperbonds form clusters,
which are at the origin of the name of "random cluster" representation.

In the original $FK$ representation of the ferromagnetic Ising model,
it is in fact the connectivity by active (or "occupied" in the original formulation) bonds
which is equivalent to the
 spin-spin 
correlations. More precisely, in the
ferromagnetic Ising model with no external field, 
\begin{eqnarray}\label{2.4.1}
\langle \omega_i \omega_j \rangle -  \langle \omega_i \rangle \langle\omega_j \rangle= P(
i \xleftrightarrow {act} j ).
\end{eqnarray}

One could hope to have a similar results, or at least an upper bound
of correlation in terms of connectivity, in greater generality. Unfortunately, 
for other Gibbs distributions (even for those which admit a directly extended
version of the FK representation) or 
for events $A, B$ which depend on more than one spin, the analogous
bound for covariances is not valid in general. Below, we make some explicit
calculation on
a very simple example: a nearest neightbor (n.n.) two body interaction model on three
aligned binary spins; in the example, couplings favor minus spins 
on the left, and plus spins on the right, both couplings involving
the middle spin; therefore, a natural RCR has no compatible active
bonds, and active bonds connectivity is zero; on the other hand, 
covariances between extreme spins are still nonzero.

\begin{example} \label{Ex1}
Take $\La=\{1,2,3\}$, 
$\Omega=\{-1,1\}^{\La}$, a two body interaction Gibbs distribution $\mu$
with interaction $\phi$ defined by
$\phi(\omega_{\{1,2\}} )= J_{12} \mathbb I_{\omega_1=\omega_2=-1}$
and $\phi(\omega_{\{2,3\}} )= J_{23} \mathbb I_{\omega_2=\omega_3=1}$.
We have
\begin{eqnarray}
\Delta \mu= \mu(\omega_1=\omega_3=1)-\mu(\omega_1=1)\mu(\omega_3=1)
&=&
\frac{(1-e^{J_{12}})(1-e^{J_{23}})}{(2(2+e^{J_{12}}+e^{J_{23}}))^2}\\
&=& \frac{(1-e^{J_{12}})(1-e^{J_{23}})}{Z^2} >0 \nonumber
\end{eqnarray}
and $Cov(\omega_1, \omega_3)= 4 \Delta \mu$.

A RCR representation has base $\nu = \nu_{12} \times \nu_{23}$,
with $\nu_{ij}$ concentrated on $\{\Omega_{\{i,j\}}, \Omega_{\{i,j\}}^*\}$
where $\Omega_{\{1,2\}}^*=\{\omega_{\{1,2\}}:\omega_1=\omega_2=-1\}$
 and $\Omega_{\{2,3\}}^*=\{\omega_{\{2,3\}}:\omega_2=\omega_3=1\}$;
 moreover, $\nu_{ij}( \Omega_{\{i,j\}}^*)=1-e^{-J_{ij}}$. In fact,
 for $\omega \in \Omega$
\begin{eqnarray*}
\sum_{\eta: \eta \sim \omega} \frac{\nu(\eta)}{\sum_{\omega', \eta'}
\nu(\eta') \mathbb I_{\eta' \sim \omega'}}
&=& \frac{e^{-J_{12} \mathbb I_{(\omega_1=\omega_2=-1)^c} 
- J_{23} \mathbb I_{(\omega_2=\omega_3=1})^c }}{Z_{\nu}} \\
&=&
\frac{e^{J_{12} \mathbb I_{(\omega_1=\omega_2=-1)} 
+ J_{23} \mathbb I_{(\omega_2=\omega_3=1)} }}{Z} = \mu(\omega).
\end{eqnarray*}
On the other hand 
\begin{eqnarray*}
P(1 
\xleftrightarrow {act} 3) &=& 
\sum_{\eta: 1 \xleftrightarrow {act} 3 \text{ in } \eta}
\nu(\eta) n_{\eta} \\
&=& \nu(\Omega_{\{1,2\}}^*,\Omega_{\{2,3\}}^*) n_{\eta} =0
\end{eqnarray*}
as $n_{\eta}=|\{\omega \in \Omega: \omega \sim \eta\}| = 0$
since 
$\Omega_{\{1,2\}}^* \cap \Omega_{\{2,3\}}^*
= \emptyset$.
So, $|Cov(\omega_1, \omega_3)| >|\Delta \mu| >0= P(1 
\xleftrightarrow {act} 3)$ and there is no upper bound of 
correlations in terms of connectivity.

\medskip

\end{example}

Clearly, there could be other RCR's of the same model for which a bound holds,
but the example shows that this does happen in general; in particular,
the example shows also that
lack of correlation, and even more  independence, does not follow
from lack of active (hyper)bond connectivity. This is an issue in the theory
of Spin Glasses, for instance, in which lack of FK bond connectivity 
does not imply uniqueness of Gibbs phase as it does in its ferromagnetic
counterpart (see \cite{N94}).

\subsection{RCR of non overlap configuration distribution and Integrated
Random Cluster distribution of active hyperbonds} 

For these reasons, we resort to 
the non overlap configuration distributions $\mu_{\phi. \La, \1}^{\sigma}$,
 and to their own RCR's. 
As $\mu_{\phi. \La, \1}^{\sigma}$ is 
$\mathcal B(\La)$-Gibbs distribution, the 
non overlap configuration distributions $\mu_{\phi. \La, \1}^{\sigma}$
is  $\mathcal B(\La)$-Gibbs on $\Omega(\sigma)$
 by Lemma \ref{2.100}; therefore, the methods shown above
produce RCR's  for each $\mu_{\phi. \La, \1}^{\sigma}$. We obtain
a collection of RCR's basis $\nu_{\phi. \La, \1}^{\sigma}$, and their
related marginals $P_{\phi. \La, \1}^{\sigma}$ over hyperbond variables
$\eta \in H^{\sigma}=\prod_{b \in \mathcal B} \Omega_b(\sigma)$,
where $ \Omega_b(\sigma)=\prod_{i \in b} F(\sigma_i).$

Notice that, by Lemma \ref{2.100}, $\mu_{\phi. \La, \1}^{\sigma}$
is $\sigma$-symmetric. Then, the RCR's can also be taken
$\sigma$-symmetric, in the sense that if $\omega_b \in \eta_b$
then also $(\sigma_b-\omega_b) \in \eta_b$; in fact, 
if $\nu$ is the base of a RCR of $\mu_{\phi. \La, \1}^{\sigma}$,
also $\nu'$ defined by $\nu'(\eta_b)= (\nu(\eta_b)+\nu(\sigma -\eta_b))/2$,
where for a set of configurations $A \subseteq \Omega_b$, $\eta_b-
A=\{\omega_b: \omega_b=\eta_b-\omega'_b \text{ for some } \omega'_b \in 
\Omega_b\}$, is a RCR for $\mu_{\phi. \La, \1}^{\sigma}$.

As the focus is on active and non active hyperbonds, we 
introduce now $H'=\prod_{b \in \mathcal B(\La)} \{0,1\}$,
$1$ standing for "active", and consider the map
$\mathscr A:H \rightarrow H'$ such that 
$
(\mathscr A(\eta))_b= \mathbb I_{
 (\eta_b \text{ is active}) } $. The measure
 $\mathscr A(P_{\phi. \La, \1}^{\sigma})$ describes
 active hyperbonds for the given $\sigma$,
 and we consider
 the   {\bf Integrated Random Cluster distribution on
 active hyperbonds }
\begin{eqnarray} \label{IRC}
\overline P_{\phi. \La, \1} (\eta')
=E_{\rho_{\phi. \La, \1} }(\mathscr A(P_{\phi. \La, \1}^{\sigma}))(\eta')
= \sum_{\sigma \in \Sigma} \mathscr A(P_{\phi. \La, \1}^{\sigma}) (\eta')
\hskip .1cm
\rho_{\phi. \La, \1} (\sigma)
\end{eqnarray}
defined on $H'$. 

The definition of $
\overline P_{\phi. \La, \1}$ is such that if an hyperbond $b$ is
fully included in the overlap region $K(\sigma)$
(in which there is only one pair of spin values satisfying
the constraints),
then  $b$ is automatically non active; this means that \eqref{IRC}
enhances the role of non active hyperbonds,
thereby making the estimates of the next section more effective.

\section{Main results} 

\subsection{Correlation inequality} 

Our main result  is a correlation inequality based on 
active hyperbond 
connectivity distributed according to the  integrated random cluster probability
$\overline P_{\phi. \La, \1}$. 
We have
\begin{theorem} \label{2.6}
For all $\La, \mathcal B \subseteq \mathcal P(\La)$,
 Gibbs probability $\mu=\mm$,
any collection of Bernoulli $\mathcal B(\La )$-RCR's
$\{\nu^{\sigma}\}_{K \subseteq \La, \alpha \in \Omega_K}$,
and any two events $A, B \subseteq \mathcal P(\Omega)$
with supports $\La_A, \La_B$, respectively, we have
\begin{eqnarray} \label{2.7}
|\mu(A \cap B)-\mu(A) \mu(B)| \leq \overline P_{\phi. \La, \1} (\La_A 
\xleftrightarrow {act} \La_B). 
\end{eqnarray}

\end{theorem}
In  words, the correlation between any pair of local events $A, B$ is bounded
by the  active hyperbond connectivity in RCR's of the non overlap configuration distributions, averaged
over the overlap configuration.

\begin{proof}
We start from a preliminary argument about $\mu^{\sigma}$,
the non overlap configuration distribution associated to $\mu=\mm$. By Lemma
 \ref{2.100}, $\mu^{\sigma}$ is $\mathcal B(\Lambda)$-Gibbs for each
 $\sigma \in \Sigma$; by Section $2.2$ and \cite{BG13}, it admits 
  Bernoulli RCR's, so the assumptions make sense.

Given a configuration $\eta$ and a vertex $i \in V$, 
we let the cluster $C(i)$ be the set of vertices connected
to $i$  by active 
hyperbonds (see Section $2.3$), each possibly consisting of just one vertex. 
%Notice that 
%the same clusters are formed under $A(\eta)$
We
denote such clusters by 
 $C_1(\eta), \dots, C_{t(\eta)}(\eta)$, with 
 $\cup_{j=1}^{t(\eta)} C_j(\eta) = \La$.
 
 Suppose that  $\eta$ is such that 
  $\La_A \centernot {\xleftrightarrow {act}} \La_B$, 
  where $\La_A $ and $\La_B$ are the supports of the given  $A$ and $B$;
  then for each $j$, $C_j(\eta)$ is connected to either $\La_A $ or $\La_B$,
  but not to both. Assume then that  $C_j(\eta)$ is connected to 
 $\La_A $ for $j=1, \dots, k$, and to $\La_B$ for $j=k+1, \dots, t(\eta)$,
 and let $Cl(A)=C_1(\eta) \cup \dots C_k(\eta)$ and
 $Cl(B)=C_{k+1}(\eta) \cup \dots C_t(\eta)(\eta)$
 indicate the cluster of $A$, and of $B$, respectively.
 Then 
 $\mathbb I_{ \omega \in A\cap B}
 =\mathbb I_{ ( \omega_{Cl(A)  } \in A)}
 \mathbb I_{ ( \omega_{Cl(B) } \in B)}$.
 
 In addition, there are no active $b$'s such that $b \cap Cl(A)
  \neq \emptyset$ and $b \cap Cl(B)\neq \emptyset$, so, for all such $b$'s, $\mathbb I_{(\omega_b \in \eta_b)}=1$.
 This justifies the third equality in the next formula.
 
 Next, recall that by the symmetry of the RCR,
 $\omega_b \in \eta_b$ if and only if $\sigma-\omega_b \in \eta_b$.
 This justifies the fourth equality below.

We then have

\begin{eqnarray} \label{mainformula1}
\mu^{\sigma}( A \cap B) &= & \sum_{\omega \in  A \cap B} \mu^{\sigma}(\omega) \nonumber\\
&=& \sum_{\omega \in  A \cap B}   \frac{1}{Z_1} \sum_{\eta \in H} \nu^{\sigma}(\eta)
\mathbb I_{ \eta \sim \omega}    
\nonumber \\ 
&\leq&  \frac{1}{Z_1} \left(
 \sum_{\eta \in H^{\sigma}: \La_A \centernot {\xleftrightarrow {act}} \La_B}  
  \nu^{\sigma}(\eta) \sum_{\omega \in \Omega_{\La}}
\mathbb I_{ \eta \sim \omega} 
\mathbb I_{  \omega \in A\cap B} \right.
  \nonumber \\
 &&\left. \quad \quad  + \sum_{\eta \in H: \La_A \xleftrightarrow {act} \La_B} 
 \sum_{\omega \in \Omega_{\Lambda}} 
  \nu^{\sigma}(\eta)
\mathbb I_{ \eta \sim \omega}  \right)
  \nonumber  \\
  &=& \frac{1}{Z_1}
 \sum_{\eta \in H^{\sigma}: \La_A \centernot {\xleftrightarrow {act}} \La_B}  
  \nu^{\sigma}(\eta) \sum_{\omega \in \Omega_{\La}}
\left( \prod_{b \subseteq Cl(A)} \mathbb I_{\omega_b \in \eta_b} 
\mathbb I_{ \omega_{Cl(A) } \in A}  \right.
 \nonumber  \\
 &&\quad \quad \quad \left.
\prod_{b \subseteq Cl(B) } \mathbb I_{\omega_b \in \eta_b} 
\mathbb I_{  \omega_{Cl(B)}  \in  B} \right)
  \nonumber \\
 &&\quad \quad  +   \sum_{\eta \in H: \La_A \xleftrightarrow {act} \La_B} 
P(\eta)
  \\
 &=&  \frac{1}{Z_1}
 \sum_{\eta \in H^{\sigma}: \La_A \centernot {\xleftrightarrow {act}} \La_B}  
  \nu^{\sigma}(\eta) 
 \sum_{\omega \in A \cap (\sigma-B)} 
\mathbb I_{ \eta \sim \omega} +   \sum_{\eta' \in H': \La_A \xleftrightarrow {act} \La_B} 
\mathscr A(P)(\eta')
    \nonumber \\
&=&\sum_{\omega \in  A \cap (\sigma-B)}   \frac{1}{Z_1} \sum_{\eta \in H} \nu^{\sigma}(\eta)
\mathbb I_{ \eta \sim \omega} +  \mathscr A(P)(\La_A \xleftrightarrow {act} \La_B ) 
    \nonumber \\
&=&\mu^{\sigma}( A \cap (\sigma-B)) + \mathscr A( P)(\La_A \xleftrightarrow {act} \La_B )
 \nonumber
\end{eqnarray}

Next, by denoting $\mu=\mu_{\phi, \La, \1}$, we have
\begin{eqnarray} 
\mu( A \cap B) &= &
(\mu \times \mu)( (A \cap B) \times \Omega_{\Lambda}) \nonumber \\
 &= & \sum_{\sigma \in \Sigma}(\mu \times \mu)((A\cap B) \times \Omega|W_{\sigma}) 
\rho_{\phi. \La, \1} (\sigma)
  \nonumber \\
&= & \sum_{\sigma \in \Sigma} \mu^{\sigma}_{\phi, \La, \1}(A\cap B) 
\rho_{\phi. \La, \1} (\sigma)
 \nonumber \\
&\leq &  \sum_{\sigma \in \Sigma} (\mu^{\sigma}_{\phi, \La, \1}(A\cap (\sigma-B)) 
+  \mathscr A(P)(\La_A \xleftrightarrow {act} \La_B ) )
\rho_{\phi. \La, \1} (\sigma)
\nonumber  \\
&= &
\sum_{\sigma \in \Sigma}(\mu \times \mu)((A\cap (\sigma-B ))\times \Omega|W_{\sigma}) 
\rho_{\phi. \La, \1} (\sigma) +  \overline P_{\phi. \La, \1} (\La_A 
\xleftrightarrow {act} \La_B)
\nonumber  \\
&= &
\sum_{\sigma \in \Sigma}(\mu \times \mu)(A \times B  |W_{\sigma}) 
\rho_{\phi. \La, \1} (\sigma) +  \overline P_{\phi. \La, \1} (\La_A 
\xleftrightarrow {act} \La_B)
\nonumber  \\
&= & 
 \mu(A) \mu(B) +  \overline P_{\phi. \La, \1} (\La_A 
\xleftrightarrow {act} \La_B)
  \nonumber 
\end{eqnarray}
The same relation holds when $B$ is replaced by $B^c$, and this proves 
\eqref{2.7}.

\end{proof}

\begin{remark}
Notice that a great number of choices has to be made in selecting a 
Bernoulli RCR of $\mu^{\sigma}$
for each $\sigma$, and the goodness of the bound depends on all of these choices. Clearly,
one can get better bounds by selecting  RCR's which use hyperbonds of small size
(see  Example \ref{Ex2} below), or give high probability
to non active hyperbond variables.
\end{remark}
\begin{remark} \label{r1}
Notice also that the  inequality in  \eqref{mainformula1} depends on 
having removed the condition that $\omega \in A \cap B$ when 
$ \La_A$ is not actively connected to $ \La_B$ in $\eta$. Maintaining that
condition would give an exact expression for the covariance of $A$ and $B$,
but the connectivity event would no longer be  measurable with 
respect to the $\eta$ variables (see  Example \ref{Ex2} below).
\end{remark}
One can get a bound on the covariance of two local random variables
by simply summing the previous on each pair of local configurations:
\begin{corollary} \label{cor1}
With the assumptions of Theorem \ref{2.6}, and two random variables $X$, $Y$,
replacing the events $A$ and $B$, depending on two disjoint finite sets
$\La_X$ and $\La_Y$, respectively, one has
$$
|Cov(X,Y)| \leq  (|F|)^{|\La_X| + |\La_Y|} \overline P_{\phi. \La, \1} (\La_X
\xleftrightarrow {act} \La_Y)
$$

\end{corollary}

\begin{example} \label{Ex2}
Continuing Example \ref{Ex1}. We apply Theorem \ref{2.6} by conditioning on 
$\sigma \in \Sigma = \{-2,0,2\}^{\{1,2,3\}}$. 
\bigskip

If $\sigma_i \neq 0$ for exactly one $i \in \{1,2,3\}$, then 
$|\Omega(\sigma)|=4$;  $\mu^{\sigma}$, however,
is symmetric under flip of the remaining spins, i.e. those located at $j$ and $k$,
with $j\neq k$, $j\neq i \neq k$, and,  therefore, only two parameters are needed,
one for $\omega_j=\omega_k$ and the other for the case $\omega_j=-\omega_k$.

If $\sigma=(0,2,0)$, however, $\mu^{\sigma}(\omega)= \frac{1}{Z} e^{J_{23}(
\omega^{(1)}_3 + \omega^{(2)}_3)}
= \frac{1}{Z}$. So, only one parameter is needed; this can be realized with a RCR 
having just a field term (i.e. bonds of size $1$), and $P^{\sigma}(1
\xleftrightarrow {act} 3)=0$.
The same occurs for $\sigma=(0,-2,0)$.

If $\sigma= (2,0,0)$, then the RCR of $\mu^{\sigma}$ has an active bond
$\eta_{2,3}$, but that does not connect $1$ and $3$, so again $P^{\sigma}(1
\xleftrightarrow {act} 3)=0$. The same occurs if $\{i:\sigma_i \neq 0\}=\{1\}$
or $\{3\}$.

If $\sigma_i \neq 0$ for more than one $i$, then no active bond is needed, as
$\mu^{\sigma}$ is binary and symmetric.

\bigskip

This leaves then only one interesting case, namely the configuration
$\tilde \sigma$ such that
$\tilde \sigma_i \equiv 0$. In this case, for $\omega=\omega^{(1)}$,
we have
\begin{eqnarray}
\mu^{\tilde\sigma} (\omega)
&=& \mu \times \mu((\omega^{(1)},-\omega^{(1)})| W_{\tilde\sigma})\\
&=&\frac{1}{Z_{\tilde \sigma}} e^{J_{12}(\mathbb I_{(\omega^{(1)}_1=\omega^{(1)}_2=-1)}
+\mathbb I_{(\omega^{(2)}_1=\omega^{(2)}_2=-1)})+
J_{23}(\mathbb I_{(\omega^{(1)}_2=\omega^{(1)}_3=1)}
+\mathbb I_{(\omega^{(2)}_2=\omega^{(2)}_3=1)}) } \nonumber \\
&=&\frac{1}{Z_{\tilde \sigma}} e^{J_{12}\mathbb I_{(\omega^{(1)}_1=\omega^{(1)}_2)}
+
J_{23}\mathbb I_{(\omega^{(1)}_2=\omega^{(1)}_3)},
 } \nonumber
\end{eqnarray}
and  $\rho(W_{\tilde\sigma})= \frac{2(e^{  J_{12} + J_{23}}+e^{  J_{12} }
+e^{  J_{23}}+1)}{Z^2}=
 \frac{Z_{\tilde \sigma}}{Z^2}$.
A Bernoulli RCR of $\mu^{\tilde \sigma}$ can now be obtained by taking  base $\nu^{\tilde\sigma}
 = \nu^{\tilde\sigma}_{12} \times \nu^{\tilde\sigma}_{23}$,
with $\nu^{\tilde\sigma}_{ij}$ concentrated on $\{\Omega_{\{i,j\}}, \Omega_{\{i,j\}}^*\}$
where $\Omega_{\{i,j\}}^*=\{\omega_{\{i,j\}}:\omega_i=\omega_j\}$, and
 moreover, $\nu^{\tilde\sigma}_{ij}( \Omega_{\{i,j\}}^*)=1-e^{-J_{ij}}$. 
 In fact,
 for $\omega \in \Omega({\tilde\sigma})$
\begin{eqnarray*}
\sum_{\eta: \eta \sim \omega} \frac{\nu^{\tilde\sigma}(\eta)}{\sum_{\omega', \eta'}
\nu^{\tilde\sigma}(\eta') \mathbb I_{\eta' \sim \omega'}}
&=& \frac{e^{-J_{12} \mathbb I_{(\omega_1 \neq \omega_2)} 
- J_{23} \mathbb I_{(\omega_2 \neq \omega_3)}}}{Z_{\nu^{\tilde\sigma}}
} \\
&=&
\frac{e^{J_{12} \mathbb I_{(\omega_1=\omega_2)}
+ J_{23} \mathbb I_{(\omega_2=\omega_3)}} }{Z_{\tilde \sigma}}  = \mu^{\tilde\sigma}(\omega),
\end{eqnarray*}
as $ Z_{\nu^{\tilde\sigma}} = Z_{\tilde \sigma}/e^{  J_{12} + J_{23}}$.
Notice that this is almost the same representation as for the one single copy
in Example \ref{Ex1}, but now the interaction has been symmetrized.

We now have
\begin{eqnarray*}
P^{\tilde\sigma}(1 \xleftrightarrow {act} 3)&=&
\frac{1}{Z_{\tilde \sigma}} \sum_{\eta : 1 \xleftrightarrow {act} 3 \text{ in } \eta}
\nu^{\tilde \sigma}(\eta)  |\{\omega : \omega \sim \eta\}|
 \\
&=& \frac{1}{Z_{\tilde \sigma}} 2 \nu^{\tilde \sigma}_{12}(\Omega_{1,2}^*) 
\nu^{\tilde \sigma}_{23}(\Omega_{2,3}^*) 
\\
&=&  \frac{1}{Z_{\tilde \sigma}} 2(1-e^{-J_{12}})(1-e^{-J_{23}}).
\end{eqnarray*}
Finally, using the value of $\Delta (\mu) 
 \mu(\omega_1=\omega_3=1)-\mu(\omega_1=1)\mu(\omega_3=1)$ computed in Example \ref{Ex1},
\begin{eqnarray*}
\overline P(1 \xleftrightarrow {act} 3)&=&
P^{\tilde\sigma}(1 \xleftrightarrow {act} 3) 
\rho(W_{\tilde\sigma})
 \\
&=&e^{ J_{12} + J_{23} } \frac{1}{Z_{\tilde \sigma}} 2(1-e^{-J_{12}})(1-e^{-J_{23}}) \frac{Z_{\tilde \sigma}}{Z^2}
 \\
&=&\frac{2(1-e^{-J_{12}})(1-e^{-J_{23}})}{Z^2} = 2| \Delta (\mu) |,
\end{eqnarray*}
which is the inequality of Theorem \ref{2.6}.

Using next the spin spin covariance computed 
in Example \ref{Ex1}, we have
$$
|Cov(\omega_1,\omega_3)| =4|\Delta (\mu)|
\leq 4 P^{\tilde\sigma}(1 \xleftrightarrow {act} 3) ,
$$
 which is the bound described in
Corollary \ref{cor1}. 

The bounds above are not sharp for the reasons mentioned in Remark \ref{r1}, and would become
equalities
 if the conditions on $\omega_1=\omega_3=1$ were kept.

\end{example}

\subsection{Conditions for extremality and uniqueness of Gibbs phases}
\begin{corollary} \label{2.81}
Suppose that  for a sequence of b.c.'s $\{ \1 \}_{\La}$, such that the finite volume 
Gibbs measures $\mm$ converge weakly as $\La$ diverges along a specific sequence, 
the following occurs: 
for each $\La_0 \subset V$ and $ \epsilon >0$ there are $\La_1, \La_2$ such that 
\begin{eqnarray} \label{3.1.1}
(a) \hskip 2cm \overline P_{\phi, \La_3, \tilde \omega_{\La_3^c}} (\La_0 
\xleftrightarrow {act} \delta \La_1) \leq \epsilon 
\end{eqnarray}
or, alternatively,
\begin{eqnarray}
(b) \hskip 2cm  p_{\epsilon} = \rho_{\phi,\La_3, \tilde \omega_{\Lambda_3^c}}
 ( P_{\phi, \La_3, \tilde \omega_{\La_3^c}}^{\sigma}(\La_0 \xleftrightarrow {act} \delta \La_1) \leq
 \epsilon ) \geq 1- \epsilon
\end{eqnarray}
for all $\La_3 \supseteq \La_2$ in the sequence of $\La$'s.
Then the weak limit $\mu$ of $\mm$ is extremal.
\end{corollary}
\begin{proof}
(a) Consider the weak limit $\mu$ of $\mm$. Consider an event $A$
 with finite support $\La_A=\La_0$
and take
 $\epsilon >0$,  
  and $\La_1$ and $\La_2$ as in the hypothesis; then for any event $B$  with support 
  $\La_B= \La_4 \subseteq \La_2$ such that $\La_4 \cap \La_1 = \emptyset$
we have
 \begin{eqnarray} \label{2.9}
 |\mu(A \cap B) -\mu(A) \mu(B)|  &\leq&  
\overline P_{\phi, \La_3, \tilde \omega_{\La_3^c}} (\La_A 
\xleftrightarrow {act} \La_B) \nonumber \\
&\leq&   \overline P_{\phi, \La_3, \tilde \omega_{\La_3^c}} (\La_0
\xleftrightarrow {act} \delta \La_1) \leq \epsilon
\nonumber
\end{eqnarray}
for all $\La_3 \supseteq \La_2$. Hence, the $\sigma$-algebra at
infinity is trivial,
which implies extremality of $\mu$
in the set $K_\phi$ of Gibbs states for $\phi$
(see, e.g., Theorem 1.11 in \cite{Ru04})  
In case (b), 
$ \overline P_{\phi, \La_3, \tilde \omega_{\La_3^c}} (\La_A 
\xleftrightarrow {act} \delta \La_1)=
 E_ {\rho_{\phi. \La_3, \tilde \omega_{\La_3^c}} }
(P^{\sigma}_{\phi, \La_3, \tilde \omega_{\La_3^c}}(\La_A 
\xleftrightarrow {act} \delta \La_1)) \\
\leq  (\epsilon p_{\epsilon} + (1-p_{\epsilon} )) \leq 2 \epsilon ,
$
so that case (a) applies.
\end{proof}
We also get a condition for uniqueness of the Gibbs state if the condition above holds for
all possible sequences of b.c.'s since $K_\phi$ is 
convex and  each element of $K_\phi$
would then be extremal. We thus have
\begin{corollary} \label{2.91}
If the conditions of Corollary \ref{2.81} hold for all sequences of b.c.'s
$\{ \1 \}_{\La}$, then the Gibbs state is unique.
\end{corollary}
To avoid technicalities the above results are stated in terms of finite volume distributions, but 
their corresponding infinite volume statements would be that 
absence of percolation of the RCR active hyperbonds in the non overlap configuration
distribution with probability one with respect to the distribution of the overlap configuration
implies uniqueness of the Gibbs state.
When sufficient conditions for such absence of percolation are expressed locally, then
our current condition closely resembles others present in the literature. Section 4 discusses these
connections. One local condition is as follows:

\begin{corollary} \label{2.200}
Consider the field $X_b^{({\phi, \La, \1})}, b \in \mathcal B$, given by 
$X_b^{({\phi, \La, \1})} (\eta) = \mathbb I_{(\eta_b \text{ is active})}$
when $\eta$ is distributed according to  $\overline P_{\phi, \La, \1}$
on $\{b: b \cap \La \neq \emptyset\}$.
If, for a sequence of b.c.'s $\1$ and each weak limit in $\La \rightarrow V$, the field 
 $X_b^{({\phi, \La, \1})}$ 
is stochastically dominated by  hyperbond occupation variables
distributed according to a Bernoulli probability $\tilde P$
on the Borel $\sigma$-algebra of $\prod_{b \in \mathcal B} \{0,1\}$,
and there is no percolation of occupied hyperbonds for $\tilde P$,
then
each weak limit of $\mm$'s is extremal.

If this happens for all b.c.'s then the Gibbs distribution is unique.

A sufficient condition for the above domination is that 
\begin{eqnarray}
p_{b}=\sup_{ \eta_{\mathcal B \setminus b}}
\overline P_{\phi, \La, \1}(\eta_b \text{ active } | \eta_{\mathcal B \setminus b}) \label{2.30}
\end{eqnarray} 
is such that there is no percolation of occupied hyperbonds when
they are independently selected with probability $p_b$.

\end{corollary}
\begin{proof}
If there is no percolation in $\tilde P$ then 
$\tilde P (\La_0 \xleftrightarrow {occupied} \delta \La_1) \rightarrow 0$
as $\La_1 $ diverges. By stochastic domination, 
$\tilde P (\La_0 \xleftrightarrow {act} \delta \La_1) \rightarrow 0$
 so that \eqref{2.8} holds, and Corollaries \ref{2.81} and \ref{2.91}
 imply the first two statements.
 
 The last statement follows from
  standard arguments in percolation theory (see, e.g. \cite{B93}, Corollary 1), as 
  \eqref{3.1.1} implies stochastic domination of $\overline P_{\phi, \La, \1}$
by a probability $\tilde P$  in which occupied hyperbonds are independently selected with probability $p_b$.

\end{proof}

\section{Applications and related works}

\subsection{ Disagreement percolation and other uniqueness criteria}
A criterium for uniqueness of Gibbs distribution has been 
introduced by Dobrushin \cite{D68} (see also \cite{S79,  DS85}), closely related to
the sufficient condition of Corollary \ref{2.200}; in some cases our method performs
better (see below).

Two copies have been considered in the works on disagreement percolation in \cite{B93, BM94, BS94}.
In the last paper there is also a correlation inequality based on two copies (see Th. 2.4 in \cite{BS94},
but it involves site percolation and only holds for the hardcore model (see below).

It is, however, interesting to relate our work to disagreement percolation in more details. 
In its first version \cite{B93}, two
independent configurations were selected, and one would focus on percolation of regions of disagreement or, 
equivalently, non overlap. The main result of \cite{B93} is that absence of disagreement percolation for
two Gibbs measures implies that they coincide. As proven in the next lemma,
absence of disagreement percolation
implies that there is a vanishing connectivity by active RCR hyperbonds in the non overlap
region for any
overlap configuration, as there is no connectivity by non overlap regions in the first place, so 
Corollary \ref{2.91} implies uniqueness of the Gibbs phase: in this respect our results uniformly
improve upon the first version of disagreement percolation. 
\begin{lemma}
If for all pairs of Gibbs measures $\mu$ and $\mu'$  in $K_{\phi}$,
the probability $\mu \times \mu'((\omega, \omega')$:  there is an infinite path of disagreement$ ) =0$,
then   condition $(b)$ of Corollary \ref{2.81} holds for all sequences of b.c.'s
$\{ \1 \}_{\La}$, and hence the Gibbs distribution is unique.
\end{lemma}
\begin{proof}
If  condition $(b)$ of Corollary \ref{2.81} does not hold then there exist a 
sequence of  b.c.'s $\{ \1 \}_{\La}$, a  set $\La_0 \subset V$, and $ \epsilon >0$ 
such that for all  $\La_1, \La_2$, $\La_1 \subseteq  \La_2$, 
 \begin{eqnarray} \label{2.8}
 p_{\epsilon} = \rho_{\phi,\La_3, \tilde \omega_{\Lambda_3^c}}( P_{\phi,\La_3, \tilde \omega_{\Lambda_3^c}}^{\sigma}(\La_0 \xleftrightarrow {act} \delta \La_1) \geq
 \epsilon ) \geq  \epsilon
\end{eqnarray}
for some $\La_3 \supseteq \La_2$. 

Notice that if, for some $\sigma$, $\La_0 \xleftrightarrow {act} \delta \La_1$
in some $\eta$, 
then necessarily there is a path of disagreement between $\La_0 $ and $ \delta \La_1$
in all the   configurations $(\omega_{\La_3},\omega'_{\La_3}) \in W_{\sigma}$
which are compatible with $\eta$. Let
$$
D(\La_0,\delta \La_1) = \{\text {there is a path of disagreement between } \La_0  \text{ and }  \delta \La_1 \},
$$
then 
\begin{eqnarray} \label{2.13}
 (\mu_{\phi, \La_3, \tilde \omega_{\La_3^c}} \times
\mu_{\phi, \La_3,  \tilde \omega_{\La_3^c}})
(D(\La_0,\delta \La_1))
&=&E_{\rho_{\phi,\La_3, \tilde \omega_{\Lambda_3^c}} }(\mu^{\sigma}_{\phi, \La_3, \tilde \omega_{\La_3^c}}
(D(\La_0,\delta \La_1)) \nonumber \\
&\geq & E_{\rho_{\phi,\La_3, \tilde \omega_{\Lambda_3^c}} }( P_{\La_3, \phi}^{\sigma}(\La_0 \xleftrightarrow {act} \delta \La_1
) )\nonumber \\
&\geq & \epsilon^2.
\end{eqnarray}

Next, consider  configurations $\omega_{\La_3 \setminus \La_1}$,
$\omega'_{\La_3 \setminus \La_1}$ and the two boundary conditions
$\omega_{\La_3 \setminus \La_1}  \tilde \omega_{\La_3^c} $ and
$\omega'_{\La_3 \setminus \La_1}  \tilde \omega_{\La_3^c}$ for Gibbs distributions in $\La_1$,
and consider 
$\mu_{\phi, \La_1, \omega_{\La_3 \setminus \La_1} \tilde \omega_{\La_3^c}} \times
\mu_{\phi, \La_1, \omega'_{\La_3 \setminus \La_1} \tilde \omega_{\La_3^c}}$. If
$$(\mu_{\phi, \La_1, \omega_{\La_3 \setminus \La_1}\tilde \omega_{\La_3^c}} \times
\mu_{\phi, \La_1, \omega_{\La_3 \setminus \La_1} \tilde \omega_{\La_3^c}})
(D(\La_0,\delta \La_1))
< \epsilon^2
$$
then  by the Gibbs formula
$$
\mu_{\phi, \La_3, \tilde \omega_{\La_3^c}}
= \int_{\Omega_{\La_3 \setminus \La_1}}   \mu_{\phi, \La_1, \omega_{\La_3 \setminus \La_1} \1}  \mu_{\phi, \La_3,\tilde \omega_{\La_3^c}}(d \omega_{\La_3 \setminus \La_1} )
$$ we have that 
$$
\mu_{\phi, \La_3,\tilde \omega_{\La_3^c}} \times
\mu_{\phi, \La_3, \tilde \omega_{\La_3^c}}
(D(\La_0,\delta \La_1) )
< \epsilon^2
$$
violating \eqref{2.13}.

Therefore, there are configurations $\omega_{\La_3 \setminus \La_1}$ ,
$\omega'_{\La_3 \setminus \La_1}$ such that 
$$(\mu_{\phi, \La_1, \omega_{\La_3 \setminus \La_1} \tilde \omega_{\La_3^c}} \times
\mu_{\phi, \La_1, \omega'_{\La_3 \setminus \La_1} \tilde \omega_{\La_3^c}})
(D(\La_0,\delta \La_1))
> \epsilon
$$ 
for some
$\La_3 \supseteq \La_2$ for each $\La_2$. By compactness and a diagonal argument, for a subsequence of 
$\La_2$'s the two sequences $\mu_{\phi, \La_1, \omega_{\La_3 \setminus \La_1} \tilde \omega_{\La_3^c}} $
and 
$
\mu_{\phi, \La_1, \omega_{\La_3 \setminus \La_1} \tilde \omega_{\La_3^c}}$ simultaneously converge in $K_\phi$ for all $\La_1$, 
so their product converges to some product of Gibbs measures $\mu \times \mu'$, for 
which 
$(\mu \times \mu')
(D(\La_0,\delta \La_1) )
> \epsilon^2 >0
$
for all $\La_1$. Hence, 
$$
\mu \times \mu'((\omega, \omega'): \text{ there is an infinite path of disagreement }) 
> \epsilon^2,
$$ contradicting the assumptions.

\end{proof}

It is interesting to notice that disagreement percolation is based on comparing
distinct boundary conditions, while our method uses the same boundary conditions
in the two copies.

The first version of disagreement percolation has been improved by using 
 optimal couplings with respect to variational distance, instead of the  independent coupling,
\cite{BM94}: there is no clear relation between the present RCR method and this improved
version of disagreement percolation, but the RCR method presented here has a more explicit geometric
interpretation, and in fact it also provides an explicit correlation bound.

\subsection{ Hard core models and complete antiferromagnets}
Hard core models are discussed in \cite{BM94}; they consist of a Gibbs measure
on $\{0,1\}^V$ given by 
$$
\mu_{a, \La, \1}(\omega_\La)= \frac{1}{Z} a^{\omega_i} 
\prod_{\langle i,j \rangle} \mathbb I_{\omega_i \omega_j=0}
\prod_{\langle i,j\rangle , i \in \La, j \not \in  \La} \mathbb I_{\omega_i \tilde \omega_j=0},
$$
i.e. $1$'s cannot be neighbor of each other.
Let us assume that $V$ is bipartite into $V_1, V_2$. Then any
for any overlap configuration $\sigma$ we must have that
for each connected component $C$ of $\La \setminus K$,
$\omega_i=1 \text{ for all }i \in V_1 \cap C, \text { and } 
\omega_i=0 \text{ for all }i \in V_2 \cap C$ or viceversa.
So, $P^{\sigma}$ is concentrated on two configurations,
and each bond of the graph is (at least part of) an active 
hyperbond. Hence, connectivity by active hyperbonds is equivalent in this 
case to connectivity by disagreement percolation in $\La \setminus K$;
our own criterium of Corollary \ref{2.91} is equivalent to 
that of \cite{BS94} (see their Proposition 3.3 and Theorem 3.4), and also
equivalent to the optimal coupling \cite{BM94} for this model.
In particular, they all  imply that there is uniqueness of the Gibbs phase if
$a < \frac{p_c}{1-p_c}$, where $p_c$ is the critical point for 
site percolation on the graph. This estimate  is better than the one
obtained with the Dobrushin Shlossman method, so also our current
one performs better than DS in this case.

\bigskip

For the complete antiferromagnet  on $V= \mathbb Z^d$, disagreement
percolation based on product coupling provided some improvement upon
previous estimates \cite{B93}. The use of optimal coupling has achieved a further
improvement \cite{BM94}, and so does our current integrated RCR method, which also
requires percolation active bonds in the disagreement or non overlap regions.
However, both optima coupling and integrated RCR do not change the 
zero temperature estimates, as the complete antiferromagnet tends to the 
hard core model as the temperature converges to zero, and 
there all estimates coincide, as discussed above.

\subsection{ Ferromagnetic Ising model on  the binary Cayley tree }

The ferromagnetic Ising model on the binary Cayley tree $(V,E)=T$
with couplings $J\geq 0$ and external field $h$,
treated here as an illustration of detailed calculations,  
has configurations $\{-1,1\}^V$ and 
$$
\mu_{(J,h), \La, \1}(\omega_\La)= \frac{1}{Z} e^{\sum_ {\langle i,j \rangle} J \omega_i \omega_j
+ \sum_{i \in \La} h \omega_i +
\sum_ {\langle i,j \rangle, i \in \La, j \not \in  \La} J\omega_i \tilde\omega_j} .
$$
A detailed description of the phase diagram is in \cite{Ge88}, Chapter 12.  
For $h(J)=max_{t \geq 0}(\log(\frac{\cosh{t+J}}{\cosh{t-J}})-t)$,
if $J \leq \log{3}/2$ and $h=h(J)=0$,  or $J > \log{3}/2$ and $|h|> h(J)$,
there is a unique Gibbs phase. We indicate by FK-RCR the Bernoulli RCR
which corresponds to the original FK representation.

Some of the Gibbs distributions $\mu$ on the Cayley tree are Markov chains,
in the sense that if $(i,j)$ is the oriented bond between two n.n. vertices $i, j \in T$,
and $\mathcal F_{(-\infty, i)}$ is the $\sigma$-algebra generated by the vertices
before $i$ in the order induced by $(i,j)$, then
$\mu(\omega_j | \mathcal F_{(-\infty, i)})= \mu(\omega_j | \omega_i)$ 
(see  \cite{Ge88}). 
\begin{corollary}
There is no percolation of active bonds in the FK-RCR of active bonds
in the non-overlap region of two independent copies of Markov chains
on the binary Cayley tree if and only if all Markov Chains are extremal.

\end{corollary}

\begin{proof}

At given $J,h$, Markov chains on the binary Cayley tree are indexed by the solutions of 
$t=h+\log(\frac{\cosh{t+J}}{\cosh{t-J}})$ and have transition matrix
$$
A= [a_{k,\ell}]= \begin{bmatrix}
\frac{e^{J-t}}{2\cosh{(J-t})}&\frac{e^{t-J}}{2\cosh{(J-t)}} \\
\frac{e^{-J-t}}{2\cosh{(J+t)}}&\frac{e^{t+J}}{2\cosh{(J+t)}}
 \end{bmatrix} 
 $$
 for $k,\ell=0,1$
 \cite{Ge88} Prop. 12.24.

Our extremality conditions of Corollary \ref{2.81} imposed on the Markov
chains give the exact calculation of the phase boundary line (although in general they
 are only sufficient conditions for uniqueness). In fact, a simple calculation shows that
 the marginal of the FK-RCR 
 for the Ising model with $h=0$ field on a tree on the active bond variables
 is just  a Bernoulli distribution,
 in which a bond is present with probability
 $p_J=\frac{(1-e^{-2J})/2}{(1-e^{-2J})/2 + e^{-2J}}= \tanh{2J}$;
 the same independence appears for the RCR in the non overlap region,
 but now $J $ is doubled, so a bond is active with probability $p_J^{NO}=\tanh{4J}$. 
 If we condition on a region in the past being connected to the vertex
 $i$ in the RCR of the non overlap region, then
 $\omega_i \neq \omega'_i$ in the two copies, and the forward process
 is independent of the past, given this information.
 In order for the bond $(i,j)$ to be active in the RCR of the non overlap region,
 it is necessary that also $\omega_j\neq \omega'_j$, and that the 
 bond is active, which occurs with probability $p_J^{NO}$.
 
 Hence, for any $\La_0$ in the past of $(i,j)$
 $$
 \overline p(J,h)=\overline P(\eta_{i,j} \text{ is active } | \La_0 \xleftrightarrow {act} i)
 =( a_{0,0} a_{11}-a_{1,0} a_{0,1}) \tanh{4 J},
 $$
 where $a_{i,j} \in A$.
 A condition for extremality of all the Markov chains is obtained, 
 following
 Corollary \ref{2.200}, by a comparison with
  the critical 
 point for independent percolation on the binary tree:
 $\overline p(J,h)\leq 1/2 $.
 Some algebraic calculations show that this occurs exactly 
 when $t=\arg max_{t \geq 0}(\log(\frac{\cosh{t+J}}{\cosh{t-J}})-t)$,
 hence at the value of $t$ which corresponds to the 
 phase boundary line in $h$.
 \end{proof}

\subsection{ Spin Glasses} 
The Edwards Anderson Spin Glass model is defined as in \eqref{2.31}.
A RCR of (a single copy of) the EA Spin Glass model is discussed in  \cite{N94}, and consists of 
$b= \{i,j\}$ for n.n. $i, j$; $H=\{ \{-1,1\}, \Omega_b\}$; and 
$\nu(\eta_{ \{i,j\}} \text{ active }) = p =1- e^{-2 J_{i,j}}$. 
Non frustration conditions appear in expressing the marginal $P$ on active bonds.

As, for each fixed overlap, the non overlap configuration distribution
is also Gibbs of the same Spin Glass type, the representation above
is also valid for the non overlaps, with doubled coupling.
Additional representations for two quenched independent  copies of EA Spin Glasses
have been discussed in Section \ref{S2.3}; in particular, we have seen
that the MNS blue-red bond representation is a typed RCR.
We see now that this representation can also be expressed in terms of 
overlap configurations, and that 
 blue bonds are either in the overlap region, or in the
non overlap region, and red bonds are in between, separating the two.

Partially numerical arguments in \cite{MNS08} suggest
the formation of two large blue clusters (one in the overlap and one in the non overlap region)
and that  multiplicity of Gibbs state (with probability one with respect to the couplings)
should appear when the two blue clusters have different densities.

Some interpretation of this behavior may come from the following
consequence of our main result.
Consider  the joint distribution of the two typed (blue and red in the MNS model)
 RCR of the quenched EA Spin Glass
$$
Q_{\bf {J}, \phi,\La, \tilde \omega_{\La^c}}(\eta^{(\alpha)}_{\La},\eta^{(\beta)}_{\La},
\omega^{(1)}_{\La}, \omega^{(2)}_{\La} )
=\frac{1}{Z} \nu^{(\alpha)} (\eta^{(\alpha)}_{\La}) \nu^{(\beta)}_{\La} (\eta^{(\beta)}_{\La} )
\mathbb I_{(\omega^{(1)}_{\La}\tilde \omega_{\La^c} , \omega^{(2)}_{\La} \tilde \omega_{\La^c} ) 
\sim \eta^{(\alpha)}_{\La} \eta^{(\beta)}_{\La}} 
$$
and for a sequence $\tilde \omega_{\La^c}$ the
 (sub)sequential limits
$
Q_{\bf {J}, \phi}$ as $\Lambda \rightarrow \infty$.
Let $A=( \text{there is percolation of } \eta^{(\alpha)} \text{ active bonds }
\{i,j\}, \text{ s.t. } \omega_i^{(1)} =- \omega_i^{(2)},  \omega_j^{(1)} =- \omega_j^{(2)})$;
$A$ corresponds to the event that there is percolation of MNS-blue bonds
in the non overlap region of a pair of Spin Glass configurations.

\begin{theorem}
If, with probability one with respect to the coupling $\bf {J}$,
for all sequences $\{\tilde \omega_{\La^c} \}_{\La}$, there is no 
MNS-blue bonds
in the non overlap region, i.e.
$Q_{\bf {J}, \phi}(A)=0$,
then the Spin Glass Gibbs state is unique  for a set of  $\bf {J}$'s
of probability one.
\end{theorem}

\begin{proof} Fix any finite set of vertices $\La_0$. 
If $Q_{\bf {J}, \phi}(A) >0$ then it is a standard procedure in percolation theory 
to select appropriate configurations around $\La_0$ such that percolation of
$\eta^{(1)}$ active bonds in the non overlap region occurs from 
$\La_0$ with positive probability. We thus assume that the $Q_{\bf {J}, \phi}$
probability of such percolation from $\La_0$ is zero for almost all $\bf {J}$'s.  Then,
for each such $\bf {J}$, and for each $\epsilon >0$, there are $\La_1, \La_2$ such
that 
$$
Q_{\bf {J}, \phi,\La_3, \tilde \omega_{\La_3^c}}(\La_0 
\xleftrightarrow {act} \delta \La_1) < \epsilon
$$
for all $\La_3 \supseteq \La_2$ in the sequence defining $Q_{\bf {J}, \phi}$.

Now, consider the map $\psi$ which transforms the variables 
representing MNS-blue bonds within $\La_1$  in the non overlap region of two configurations,
namely $(\eta^{(\alpha)}, \omega^{(1)}, \omega^{(2)})$ as defined in Section \ref{S2.3},
into active bonds in the Bernoulli RCR  of the non overlap configuration distribution, 
defined as follows.
Let $\{i,j\}=b \subseteq \La_1$; with $\eta^{(\alpha)}_{i,j}$ having values either $\Omega^{(1)}_{\{i,j\}}
\times \Omega^{(2)}_{\{i,j\}}$ or the set $\{(\omega^{(1)},\omega^{(1)}):
J_{i,j}\omega^{(1)}_i \omega^{(1)}_j=1, J_{i,j}\omega^{(2)}_i \omega^{(2)}_j=1\}$,
then 
$$(\psi(\eta^{(\alpha)}_{\La_1},\omega^{(1)}_{\La_1},\omega^{(2)}_{\La_1}) )_{i,j}
= 
\begin{cases}
\Omega_{i,j} & \mbox{if } \eta^{(\alpha)}_{i,j} = \Omega^{(1)}_{\{i,j\}}
\times \Omega^{(2)}_{\{i,j\}} \\
&\text{ and }\omega^{(1)}_i =-\omega^{(2)}_i,
\omega^{(1)}_j =-\omega^{(2)}_j.
\\
\{ \omega_{i,j} : J_{i,j} \omega_i \cdot \omega_j =1\}
& \mbox{ otherwise }
\end{cases} 
$$
%If $\Omega_{\La_1}^{diff}$ is the set of pairs of configurations
%$(\omega^{(1)}_{\La_1},\omega^{(2)}_{\La_1})$ which are 
%opposite in $\La_1$, i.e. $ \omega^{(1)}_i=-\omega^{(2)}_i$
%for all $i \in \La_1$,
%then $\psi$ is $1$-$1$ in $\Omega_{\La_1}^{diff}$; in fact; if 
%$\omega^{(1)}$ and $\omega^{(2)}$ are different in both $i $ and $j$,
%then either both $J_{i,j} \omega^{(\ell)}_i \cdot \omega^{(\ell)}_j =1$, 
%$\ell=1,2$, or they are both $-1$; so, $ \eta^{(1)}_{i,j}$
% has always the two choices of being
%active or not (instead of being forced to be inactive), and $\psi$ maps 
%$1-1$ to the active or inactive $\eta$'s in the RCR of the non overlap region.

We show here below that 
$$
Q_{\bf {J}, \phi,\La_3, \tilde \omega_{\La_3^c}}(\eta^{(\alpha)}_{\La_1},\omega^{(1)}_{\La_1},\omega^{(2)}_{\La_1}) 
=\overline P_{\bf {J}, \phi,\La_3, \tilde \omega_{\La_3^c}}(\psi(\eta^{(\alpha)}_{\La_1},\omega^{(1)}_{\La_1},\omega^{(2)}_{\La_1}) ).
$$
Then, 
\begin{eqnarray*}
\epsilon &\geq &Q_{\bf {J}, \phi,\La_3, \tilde \omega_{\La_3^c}}(\La_0 
\xleftrightarrow {act} \delta \La_1) \\
&=&
\overline P_{\bf {J}, \phi,\La_3, \tilde \omega_{\La_3^c}}(\psi(\La_0 
\xleftrightarrow {act} \delta \La_1))\\
&=& 
\overline P_{\bf {J}, \phi,\La_3, \tilde \omega_{\La_3^c}}(\La_0 
\xleftrightarrow {act} \delta \La_1)
\end{eqnarray*}
for all $\La_3 \supseteq \La_2$,
and hence by Part $(a) $ of Corollary  \eqref{2.81}, the Gibbs state is unique.

To conclude, we have the following. For given $\sigma$ let
$\La_{\sigma} = \{i: \sigma_1=0\}$ and 
$\setminus \La_{\sigma} = \mathcal B(\La)
\setminus \mathcal B(\La_{\sigma} )$

\begin{eqnarray*} 
&& Q_{\bf {J}, \phi,\La_3, \tilde \omega_{\La_3^c}}
(\overline \eta^{(\alpha)}_{\La_1},\omega^{(1)}_{\La_1},\omega^{(2)}_{\La_1})\\
&&\quad
\sum_{\eta^{(\alpha)}_{\La_3}, \eta^{(\beta)}_{\La_3}: \eta^{(\alpha)}_{\La_1}=\overline \eta^{(\alpha)}_{\La_1}} \\
&&\hskip 0.5cm
\sum_{\omega^{(1)}_{\La_3},\omega^{(2)}_{\La_3}:\omega^{(1)}_{\La_1}=-\omega^{(2)}_{\La_1}}
Q_{\bf {J}, \phi,\La_3, \tilde \omega_{\La_3^c}}
(\eta^{(\alpha)}_{\La_3}, \eta^{(\beta)}_{\La_3},\omega^{(1)}_{\La_3},\omega^{(2)}_{\La_3})\\
&=&
\sum_{\eta^{(\alpha)}_{\La_3}, \eta^{(\beta)}_{\La_3}: \eta^{(\alpha)}_{\La_1}=\overline \eta^{(\alpha)}_{\La_1}} \\
&&\hskip 0.5cm
\sum_{(\omega^{(1)}_{\La_3},\omega^{(2)}_{\La_3})\sim(\eta^{(\alpha)}_{\La_3}, \eta^{(\beta)}_{\La_3}):\omega^{(1)}_{\La_1}=-\omega^{(2)}_{\La_1}}
\frac{1}{Z} \nu^{(\alpha)} (\eta^{(\alpha)}_{\La_3}) \nu^{(\beta)}_{\La} (\eta^{(\beta)}_{\La_3} )\\
&=& \sum_{\sigma: \sigma_{\La_1} \equiv 0 }
\frac{1}{Z} \prod_{\{i,j\} \subseteq \La_{\sigma}} (1-e^{-4 J_{i,j}} )\mathbb I_{\overline \eta_{i,j} 
\text{(is active)}} +e^{-4 J_{i,j}} \mathbb I_{\overline \eta_{i,j} 
\text{(is not active)}}\\
&&\hskip 0.5cm
\sum_{\eta^{(\alpha)}_{\setminus \La_{\sigma}}, \eta^{(\beta)}_{\setminus \La_{\sigma}}}
\prod_{\{(i,j\} \not \subseteq \La_{\sigma}, \omega^{(1)}_i \cdot \omega^{(2)}_i=
 \omega^{(1)}_j \cdot \omega^{(2)}_j }
 (1-e^{-4 J_{i,j}} )\mathbb I_{\overline \eta_{i,j} 
\text{(is active)}} +e^{-4 J_{i,j}} \mathbb I_{\overline \eta_{i,j} 
\text{(is not active)}}\\
&&\hskip 0.9cm \prod_{\{(i,j\} \not \subseteq \La_{\sigma}, \omega^{(1)}_i \cdot \omega^{(2)}_i=
 -\omega^{(1)}_j \cdot \omega^{(2)}_j }
 (1-e^{-2 J_{i,j}} )\mathbb I_{\overline \eta_{i,j} 
\text{(is active)}} +e^{-2 J_{i,j}} \mathbb I_{\overline \eta_{i,j} 
\text{(is not active)}}\\
&&\hskip 1.2cm   \times
|\{(\omega^{(1)}_{\La_3},\omega^{(2)}_{\La_3}):
(\omega^{(1)}_{\La_3},\omega^{(2)}_{\La_3}) \sim 
\eta^{(1)}_{ \La_{\sigma}}, \eta^{(1)}_{\setminus \La_{\sigma}},
 \eta^{(2)}_{\setminus \La_{\sigma}} \}|\\
 &=&
  \sum_{\sigma: \sigma_{\La_1 \equiv 0} }
\frac{1}{Z} \prod_{\{i,j\} \subseteq \La_{\sigma}} (1-e^{-4 J_{i,j}} )\mathbb I_{\overline \eta_{i,j} 
\text{(is active)}} +e^{-4 J_{i,j}} \mathbb I_{\overline \eta_{i,j} 
\text{(is not active)}}\\
&&\hskip 0.6cm \times
|\{(\omega^{(1)}_{\La_3},\omega^{(2)}_{\La_3}):
(\omega^{(1)}_{\La_3},\omega^{(2)}_{\La_3}) \sim 
\eta^{(1)}_{ \La_{\sigma}}, \eta^{(1)}_{\setminus \La_{\sigma}},
 \eta^{(2)}_{\setminus \La_{\sigma}} \}|
 \\ &&\hskip 1.5cm 
\times \rho( (\omega^{(1)}_{\setminus \La_{\sigma}},\omega^{(2)}_{\setminus \La_{\sigma}})
: \omega^{(1)}_i = \omega^{(2)}_i \text{ for all } i \in \La_3 \setminus \La_{\sigma}) \\
&=& E_{\rho} (\nu^{\rho} (\psi(\overline \eta^{(1)}_{\La_1},\omega^{(1)}_{\La_1},\omega^{(2)}_{\La_1}) )
\\
&=& \overline P(\psi(\overline \eta^{(1)}_{\La_1},\omega^{(1)}_{\La_1},\omega^{(2)}_{\La_1});
\end{eqnarray*}
the third equality follows from the fact that the sum is independent of  
$(\omega^{(1)}_{\La_3},\omega^{(2)}_{\La_3})$ as the cross interactions between 
the non overlap configuration in $\La_{\sigma}$ and the 
overlap configuration in $\La_3 \setminus \La_{\sigma}$ is always zero,
since  for $i \in \La_{\sigma}$ and $j \not \in \La_{\sigma}$
we have 
$J_{i,j}( \omega^{(1)}_i \cdot \omega^{(1)}_j  + \omega^{(2)}_i  \cdot \omega^{(2)}_j)
=J_{i,j}( \omega^{(1)}_i \cdot \omega^{(1)}_j  - \omega^{(1)}_i  \cdot \omega^{(1)}_j)
= 0$.

This finishes the proof.

\end{proof}

The last theorem suggests that the observed unbalance
in blue cluster densities at the phase transition could be caused by the onset of percolation
of the blue cluster in the non overlap region.
Restricting to dimension $2$, it is
conceivable that planar geometric constraints  prevent the formation of 
a percolating blue cluster in the non overlap region and this could
lead to a proof of the absence of phase transition at any finite temperature
in the two-dimensional  EA Spin Glasses.

 %*****************************************

\bigskip

Contact address:
NYU Abu Dhabi 
Saadiyat Island
P.O Box 129188
Abu Dhabi, UAE

email: ag189@nyu.edu

\end{document}